\UseRawInputEncoding
\documentclass[11pt]{amsart}
\usepackage{amsfonts}
\usepackage{}
\usepackage{mathrsfs}
\usepackage{bbm}
\usepackage{amssymb}
\usepackage{indentfirst}
\usepackage{latexsym,amsfonts,amssymb,amsmath}
\pdfoutput=1
\newtheorem{theorem}{Theorem}[section]
\newtheorem{lemma}[theorem]{Lemma}
\newtheorem{corollary}[theorem]{Corollary}

\newtheorem{fact}[theorem]{Fact}
\theoremstyle{definition}
\newtheorem{definition}[theorem]{Definition}
\newtheorem{proposition}[theorem]{Proposition}
\theoremstyle{remark}

\newtheorem{example}[theorem]{Example}
\begin{document}

\title
{Weak countability axioms on the quotient spaces of topological gyrogroups}

\author{Ying-Ying Jin}\thanks{}
\address{(Y.-Y. Jin) General Education Department, Guangzhou Polytechnic University, Guangzhou 511483, P.R. China} \email{yingyjin@163.com, jinyy@gzpyp.edu.cn}

\author{Yi-Ting Wang}\thanks{}
\address{(Y.-T. Wang) School of Mathematics and Computational Science, Wuyi University, Jiangmen 529020, P.R. China} \email{2205616228@qq.com}

\author{Li-Hong Xie*}\thanks{* The corresponding author.}
\address{(L.-H. Xie) School of Mathematics and Computational Science, Wuyi University, Jiangmen 529020, P.R. China} \email{yunli198282@126.com}


\thanks{
This work is supported by the Natural Science Foundation of Guangdong
Province under Grant (Nos. 2021A1515010381; 2020A1515110458). The Innovation Project of Department of Education of Guangdong Province (No. 2022KTSCX145), and the Natural Science Project of Jiangmen City (No. 2021030102570004880), Scientific research project of Guangzhou Panyu Polytechnic  (No. 2022KJ02)}

\subjclass[2010]{primary 54H99; secondary 54D30, 54D45, 54D20, 	54B15, 54C10}

\keywords{Topological gyrogroup; Quotient space; Neutral subgyrogroup; Weak countability axiom}

\begin{abstract}
In this paper, we mainly prove that if $H$ is a
closed strong subgyrogroup of a strongly
topological gyrogroup $G$ and $H$ is neutral,
then
  (1) \( G/H \) is biradial if and only if \( G/H \) is nested;
  (2) \( G/H \) is metrizable if and only if \( G/H \) is a biradial space with countable pseudocharacter;
  (3) \( G/H \) is metrizable if and only if \( G/H \) has countable $cn$-character, given that \( G/H \) has the Baire property.
\end{abstract}

\maketitle

\section{Introduction}

Indeed, a topological group \( G \) is a group equipped with a topology that ensures the continuity of both its multiplication operation and the operation of taking inverses within \( G \). Moving forward, let us examine a topological group \( G \) and a closed subgroup \( H \) of \( G \). It is acknowledged that the quotient space \( G/H \) does not maintain the structure of a topological group.
A space $X$ is designated as a coset space if it is homeomorphic to a quotient space $G/H$ formed by the left cosets of a topological group $G$, where $H$ is a closed subgroup. It is important to note that a coset space does not necessarily resemble a topological group in its structure.
Although coset spaces themselves might not be considered topological groups, they exhibit homogeneity and Tychonoff properties. According to Bourbaki \cite{Bou}, every homogeneous, zero-dimensional compact Hausdorff space is a coset space.
As a natural extension of topological groups, it is logical to explore which results that hold for topological groups can be generalized to coset spaces.
Fern¨¢\'{a}ndez, S\'{a}nchez, and Tkachenko \cite[Corollary 3.6]{Fer} demonstrated that for a broad class of coset spaces, metrizability is equivalent to first-countability.
They proved the following:
Let \( H \) be a closed neutral subgroup of a topological group \( G \). Then \( G/H \) is metrizable if and only if \( G/H \) is first-countable.

Birkhoff-Kakutani has proved in \cite[Theorem 3.3.12]{Arha} that
a topological group $G$ is metrizable if and only if $G$ is first-countable.
Based on the Birkhoff-Kakutani Theorem, Malykhin put forward a natural question regarding the extent to which first-countability can be weakened (refer to
\cite{Arha2, Mo, Va}). This problem has drawn considerable attention from numerous scholars. These scholars have delved into various forms of weak first-countability in topological groups, including weakly first-countable spaces, bisequential spaces, csf-countable spaces, and Fr\'{e}chet-Urysohn spaces
\cite{Arha3, Ba1, Gab, Liu, No, Ny, Sh}.
It is worth noting that Hru$\check{s}$\'{a}k and Ramos-Garc\'{\i}a \cite[Theorem 1.6]{Hr} built a model of ZFC. In this model, every separable Fr\'{e}chet-Urysohn topological group is metrizable. As a result, they resolved a problem that Malykhin had presented in 1978.
Inspired by Fern¨¢\'{a}ndez, S\'{a}nchez, and Tkachenko's result, Ling et al. studied the weak countability axioms of coset spaces \cite{Ling}.

Gyrogroups, being a generalization of groups, exhibit a weaker algebraic structure compared to groups. Their foundational principles are largely rooted in the concept of relativistically admissible velocity space \cite{Ung}. Notably, gyrogroups have found applications across a diverse array of disciplines including mathematical physics \cite{Pa}, group theory \cite{Fo}, and abstract algebra \cite{Suk} and so on.
In 2017, Atiponrak introduced the theory of topological gyrogroups, presenting them as an extension of topological groups \cite{Atip}. A pivotal contribution came in 2019 from Cai, Lin, and He \cite{Cai}, who demonstrated that every topological gyrogroup is rectifiable, thereby establishing that any Hausdorff topological gyrogroup is metrizable if and only if it satisfies the first-countability property. This finding provided a positive resolution to a problem posed in \cite{Atip}.
Zhang et al.
in \cite{zh} proved that if $(G, \tau, \oplus)$ is a strongly topological gyrogroup and
$H$ is a closed strong subgyrogroup of $G$, then $G/H$ is $\kappa$-Fr\'{e}chet-Urysohn if and only
if $G/H$ is strongly $\kappa$-Fr\'{e}chet-Urysohn under the condition that $H$ is neutral.
Liu and Zhang \cite{Liu1} demonstrated that metrizability and first-countability are equivalent for a broad class of quotient spaces.
By examining neutral subgyrogroups, they derived the following metrization theorem for quotient spaces,
which extends the Fern¨¢\'{a}ndez, S\'{a}nchez, and Tkachenko's result to a wider range of topological spaces.

\begin{theorem}\cite{Liu1}\label{the1.1s}
Let $H$ be a closed strong subgyrogroup of a strongly topological gyrogroup $(G, \tau, \oplus)$. If $H$ is
neutral, then $G/H$ is metrizable if and only if $G/H$ is first-countable.
\end{theorem}

In this paper, we mainly consider
weak countability axioms of the quotient spaces of strongly topological gyrogroups.
All spaces in this paper are Hausdorff unless stated otherwise,
the symbols $\omega$ and $N$ denote the first infinite ordinal and
the set of all positive integers respectively.

\section{Definitions and preliminaries}
This section contains necessary definitions and background of gyrogroups and paratopological (topological) gyrogroups.
Their basic algebraic and topological properties are included as well.
For notation and terminology not explicitly mentioned here, readers are encouraged to refer to \cite{Arha, En89}.

Let $G$ be a nonempty set, and let $\oplus  : G  \times G \rightarrow G $ be a binary operation on $G $. Then the pair $(G, \oplus)$ is
called a {\it groupoid.}  A function $f$ from a groupoid $(G_1, \oplus_1)$ to a groupoid $(G_2, \oplus_2)$ is said to be
a groupoid homomorphism if $f(x_1\oplus_1 x_2)=f(x_1)\oplus_2 f(x_2)$ for any elements $x_1, x_2 \in G_1$.  In addition, a bijective
groupoid homomorphism from a groupoid $(G, \oplus)$ to itself will be called a groupoid automorphism. We will write $Aut (G, \oplus)$ for the set of all automorphisms of a groupoid $(G, \oplus)$.

The notion of a gyrogroup was first identified by Ungar during his research on Einstein's velocity addition \cite{Ung}.

\begin{definition}\cite[Gyrogroups]{Ung}\label{Def:gyr}
 Let $(G, \oplus)$ be a nonempty groupoid. We say that $(G, \oplus)$ or just $G$
(when it is clear from the context) is a gyrogroup if the followings hold:
\begin{enumerate}
\item[($G1$)] There is an identity element $0 \in G$ such that
$$0\oplus x=x=x\oplus 0\text{~~~~~for all~~}x\in G.$$
\item[($G2$)] For each $x \in G $, there exists an {\it inverse element}  $\ominus x \in G$ such that
$$\ominus x\oplus x=0=x\oplus(\ominus x).$$
\item[($G3$)] For any $x, y \in G $, there exists an {\it gyroautomorphism} $\text{gyr}[x, y] \in Aut(G,  \oplus)$ such that
$$x\oplus (y\oplus z)=(x\oplus y)\oplus \text{gyr}[x, y](z)$$ for all $z \in G$;
\item[($G4$)] For any $x, y \in G$, the equation $\text{gyr}[x \oplus y, y]=\text{gyr}[x, y]$ holds.
\end{enumerate}
\end{definition}

\begin{definition}\cite{Ung}\label{defbox}
Let $(G,\oplus)$ be a gyrogroup with gyrogroup operation (or,
addition) $\oplus$. The gyrogroup cooperation (or, coaddition) $\boxplus$ is a second
binary operation in $G$ given by the equation
$a\boxplus b=a\oplus \text{gyr}[a,\ominus b]b$ for all $a, b\in G$.
The groupoid $(G, \boxplus)$ is called a cogyrogroup, and is said to
be the cogyrogroup associated with the gyrogroup $(G, \oplus)$.

Replacing $b$ by $\ominus b$, we have the identity
$a\boxminus b=a\ominus \text{gyr}[a,b]b$ for all $a, b\in G$, where we use the obvious notation, $a\boxminus b = a\boxplus(\ominus b)$.
\end{definition}

Theorem \ref{the1.3} below summarizes some fundamental algebraic properties of gyrogroups.

\begin{theorem}\cite{Ung}\label{the1.3}
Let $(G, \oplus)$ be a gyrogroup. Then, for any $a, b, c\in G$ we have
\begin{enumerate}
\item[(1)] $(a\oplus b)\oplus c=a\oplus(b\oplus \text{gyr}[b, a]c);$\hfill{Right Gyroassociative Law}
\item[(2)] $a\oplus (b\oplus c)=(a\oplus b)\oplus \text{gyr}[a, b]c;$\hfill{Left Gyroassociative Law}
\item[(3)] $\text{gyr}[a, b]=\text{gyr}[a, b\oplus a];$\hfill{Right Loop Property}
\item[(4)] $\text{gyr}[a, b]=\text{gyr}[a\oplus b, b];$\hfill{Left Loop Property}
\item[(5)] $(\ominus a)\oplus(a\oplus b)= b$;
\item[(6)] $(a\ominus b)\boxplus b= a$;
\item[(7)] $(a\boxminus b)\oplus b= a$;
\item[(8)] $\text{gyr}[a, b](c)=\ominus(a\oplus b)\oplus (a\oplus (b\oplus c))$;
\item[(9)] $\ominus(a\oplus b)=\text{gyr}[a, b](\ominus b\ominus a)$;\hfill{Gyrosum Inversion}
\item[(10)] $\text{gyr}[a, b](\ominus x)=\ominus \text{gyr}[a, b]x$;
\item[(11)] $\text{gyr}^{-1}[a, b]=\text{gyr}[b, a]$; \hfill{Inversive symmetry}
\item[(12)] $\ominus(a\boxplus b)= (\ominus b)\boxplus(\ominus a)$; \hfill{The Cogyroautomorphic Inverse Theorem}
\item[(13)] $\text{gyr}[\ominus a, \ominus b]=\text{gyr}[a, b]$; \hfill{Even symmetry}
\item[(14)] $\text{gyr}[a, 0]=\text{gyr}[0, b]=\text{id}_G$.
\end{enumerate}
\end{theorem}

Generally, gyrogroups do not adhere to the associative property. However, they comply with the left and right gyroassociative laws, denoted as Proposition \ref{the1.3} (1) and (2), respectively.
Observe that a group is a type of gyrogroup $(G, \oplus)$ where \text{gyr}$[x, y]$ acts as the identity map for all $x, y$ in $G$. This indicates that gyrogroups serve as a natural extension of the concept of groups.
Specifically, gyrogroups that extend the principles of Abelian groups are assigned a distinct designation.

\begin{proposition} \label{the3.22s}
Let $(G, \oplus)$ be a  gyrogroup. Then
$a\boxplus b=b\oplus((\ominus b\oplus a)\oplus b)$
for all $a, b\in G$.
\end{proposition}

\begin{proof}
For all $a, b\in G$ we have
\begin{align*}
b\oplus((\ominus b\oplus a)\oplus b)&=(b\oplus(\ominus b\oplus a))\oplus \text{gyr}[b,\ominus b\oplus a]b \quad \text{by the left gyroassociative law~}
\\&=a\oplus \text{gyr}[a,\ominus b\oplus a]b \quad \text{by a left cancellation and a left loop property~}
\\&=a\oplus \text{gyr}[a,\ominus b]b \quad \text{by a right loop property~}
\\&=a\boxplus b.  \quad\quad\quad \text{by Definition \ref{defbox}~}
\end{align*}
\end{proof}

\begin{definition}\cite[Proposition 14]{Suk3}
Assume that $(G, \oplus)$ is a gyrogroup and $H$ is a nonempty subset of $G$. Then $H$ is a \textit{subgyrogroup} iff $\ominus a \in H$ and $a \oplus b \in H$ for all $a, b \in H$.
\end{definition}

\begin{definition}\cite[L-subgyrogroup]{Suk3} Let $(G,\oplus)$ be a gyrogroup and $H$ a subgyrogroup of $G$. Then $H$ is an \textit{L-subgyrogroup}, if $\text{gyr}[a,h](H)=H$ holds for all $a\in G$ and $h\in H$.
\end{definition}
 Let $A$ be a subset of a gyrogroup $G$. We call that $A$ is {\it gyr-invariant} if $\text{gyr}[x,y]A=A$ for all $x,y\in G$. Furthermore, a subgyrogroup $H$ of a given gyrogroup $(G,\oplus)$ is said to be a \textit{strong subgyrogroup} (\cite[Definition 3.9]{B1}), if $H$ is gyr-invariant. Note that each strong subgyrogroup is an L-subgyrogroup.

The family of open neighborhoods of the
identity 0 in $G$ will be denoted by $\mathcal{U}$.
No separation restrictions on the topological spaces considered in this paper are imposed unless we mention them explicitly.
Moreover, the set of all positive integers denoted by $\mathbb{N}$ and the first infinite ordinal denoted by $\omega$.
We are prepared to present the definition of gyrogroups as follows.

\begin{definition}\cite{Atip}\label{def2.11}
A triple $(G, \tau,  \oplus)$ is called a {\it topological gyrogroup} if and only if
\begin{enumerate}
\item[(1)] $(G, \tau)$ is a topological space;
\item[(2)] $(G, \oplus)$ is a gyrogroup;
\item[(3)] The binary operation $\oplus:G \times G\rightarrow G$ is continuous where $G\times G$ is endowed with the product topology
and the operation of taking the inverse $\ominus(\cdot ) : G  \rightarrow G $, i.e. $x\rightarrow\ominus x$, is continuous.
\end{enumerate}
\end{definition}

If a triple $( G, \tau,  \oplus)$ satisfies the first two conditions and its binary operation is continuous, we call such
triple a {\it paratopological gyrogroup} \cite{Atip1}. Sometimes we will just say that $G$ is a topological gyrogroup (paratopological gyrogroup) if the binary operation and the topology are clear from the context.

\begin{example}\cite[Example 3]{Atip}\label{ex13}
Let's examine the set of all Einsteinian velocities, characterized as follows:
$$\mathbb{R}^3_{c} = \{\mathbf{v} \in \mathbb{R}^3 : \|\mathbf{v}\| < c\}.$$
In this context, $c$ represents the speed of light in a vacuum, and $\|\mathbf{v}\|$ denotes the Euclidean norm of a vector $\mathbf{v}$ in $\mathbb{R}^3$. This set, being a subset of a topological space (namely, $\mathbb{R}^3$ with its standard topology), naturally forms a topological space itself. We then introduce the operation of Einstein addition, $\oplus_E: \mathbb{R}^3_{c} \times \mathbb{R}^3_{c} \to \mathbb{R}^3_{c}$, for any vectors $\mathbf{u}, \mathbf{v}$ in $\mathbb{R}^3_{c}$, as
$$\mathbf{u} \oplus_E \mathbf{v} = \frac{1}{1 + \frac{\mathbf{u} \cdot \mathbf{v}}{c^2}} \left( \mathbf{u} \oplus \frac{1}{\gamma_{\mathbf{u}}} \mathbf{v} + \frac{1}{c^2} \frac{\gamma_{\mathbf{u}}}{1 + \gamma_{\mathbf{v}}} (\mathbf{u} \cdot \mathbf{v}) \mathbf{u} \right),$$
where $\mathbf{u} \cdot \mathbf{v}$ is the standard dot product of vectors in $\mathbb{R}^3$, and the gamma factor $\gamma_{\mathbf{u}}$ within the $c$-ball is defined by
$$\gamma_{\mathbf{u}} = \frac{1}{\sqrt{1 - \frac{\mathbf{u} \cdot \mathbf{u}}{c^2}}}.$$
Moreover, given the standard topology derived from $\mathbb{R}^3$, it is evident that the operation $\oplus_E$ is continuous.
Additionally, it has been established that the inverse of any $\mathbf{u}$ in $\mathbb{R}^3_c$ is $-\mathbf{u}$, thereby ensuring that the inversion operation is continuous as well. Thus, $\mathbb{R}^3_{c}$ constitutes a topological gyrogroup, distinguishing itself from being a topological group, a semitopological group, or a paratopological group, in a rephrased manner.
\end{example}

Assume that $A, B$ are subsets of a gyrogroup $(G, \oplus)$, then $A \oplus B$ and $\ominus A$ are defined by $A \oplus B = \{a \oplus b : a \in A, b \in B\}$ and $\ominus A = \{\ominus a : a \in A\} = (\ominus)^{-1}(A)$.
Next, we present the definition of strongly topological gyrogroups, which is crucial for the content of this paper.

\begin{definition}\cite{BL}\label{defst}
Let $G$ be a topological gyrogroup. We say that $G$ is a strongly topological gyrogroup if
for each neighborhood $U$ at the identity there exists a gyr-invariant neighborhood $V$ such that $V\subseteq U$.
\end{definition}
For a paratopological gyrogroup $( G, \tau,  \oplus)$, we called $( G, \tau,  \oplus)$ a {\it strongly paratopological gyrogroup} if
there exists a neighborhood base $\mathcal{U}$ at the identity element consisting of gyr-invariant subsets in $G$ \cite{JX3}.

\begin{proposition}\label{prop2.10s}
Let $(G,\oplus)$ be a gyrogroup and $V$ a gyr-invariant subset in $G$. Then $(W\oplus U)\oplus V=W\oplus (U\oplus V)$  holds for each
$W,U\subseteq G$.

\end{proposition}
\begin{proof}
By the definition of gyr-invariant subset, we can get
 $(W\oplus U)\oplus V=W\oplus (U\oplus\bigcup_{u\in U, w\in W} \text{gyr}[u,w]V)=W\oplus( U\oplus V)$.
\end{proof}

\begin{definition}\cite{B1}
A subgyrogroup $H$ of a topological gyrogroup $(G,\tau,\oplus)$ is called \textit{inner neutral} (resp. \textit{outer neutral}) if for every neighborhood $U$ of $0$ in $G$, there is an open neighborhood $V$ of $0$ in $G$ such that $H\oplus V \subseteq U\oplus H$ (resp. $V\oplus H \subseteq H\oplus U$).
\end{definition}
We say that $H$ is \textit{neutral} if $H$ is both inner neutral and outer neutral in $G$.


Consider a topological gyrogroup $(G, \tau, \oplus)$ with $H$ being an $L$-subgyrogroup of $G$.
As derived from \cite[Theorem 20]{Suk3}, the set $G/H = \{a \oplus H : a \in G\}$ constitutes a division of $G$ into distinct parts. The function $\pi$, which maps each element $a\in G$ to the coset $a \oplus H$ onto $G/H$, ensures that the preimage $\pi^{-1}(\pi(a))=a \oplus H$. We refer to $\tau(G)$ as the topology on $G$. In defining a topology $\widetilde{\tau}$ on the set of left cosets $G/H$ of the gyrogroup $G$, we employ $\widetilde{\tau}=\{O \subset G/H : \pi^{-1}(O)\in \tau(G)\}$, with $\tau(G)$ indicating the established topology of $G$.
Based on \cite[Theorem 3.13]{B1}, it can be deduced that the mapping $\pi$ is open and that the quotient space $G/H$ is $T_1$ homogeneous whenever $H$ is identified as a closed strong subgyrogroup.

\begin{proposition}\label{pro2.28}\cite{JX2}
Let $(G, \tau,\oplus)$ be a paratopological gyrogroup and $H$ a $L$-subgyrogroup of $G$.
Then the natural homomorphism
$\pi$ from a paratopological gyrogroup $G$ to its quotient topology on $G/H$ is an open and continuous mapping.
\end{proposition}

\begin{corollary}\cite{Liu1}\label{cor2.17}
Assume that $(G, \tau, \oplus)$ is a strongly topological gyrogroup and $H$ is a closed strong subgroup of $G$, then the quotient space $G/H$ is complete regular if $H$ is neutral.
\end{corollary}

The following definitions can also be found in \cite{Ling}. For the sake of completeness of this paper, they are listed here.

\begin{definition}
Let $X$ be a space.

(1) $X$ is called a \textit{weakly first-countable space} if for each point $x \in X$ it is possible to assign a sequence $\{B(n,x): n \in \mathbb{N}\}$ of subsets of $X$ containing $x$ in such a way that $B(n+1,x) \subseteq B(n,x)$ and so that a set $U$ is open if, and only if, for each $x \in U$ there exists $n \in \mathbb{N}$ such that $B(n,x) \subseteq U$.

(2) $X$ is called a \textit{sequential space} if for each non-closed subset $A \subset X$, there are a point $x \in X \setminus A$ and a sequence in $A$ converging to $x$ in $X$.

(3) $X$ is a \textit{k-space} if each $A \subseteq X$, then $A$ is closed in $X$ if and only if the intersection of $A$ with any compact subspace $Z$ of $X$ is closed in $Z$.

(4) $X$ is called a \textit{Fr\'{e}chet-Urysohn space} if for any subset $A \subseteq X$ and $x \in \overline{A}$, there is a sequence in $A$ converging to $x$ in $X$.

(5) $X$ is called a \textit{strongly Fr\'{e}chet-Urysohn space} if the following condition is satisfied:

(SFU) For each $x \in X$ and every sequence $\xi = \{A_n: n \in \mathbb{N}\}$ of subsets of $X$ such that $x \in \bigcap_{n \in \mathbb{N}} \overline{A_n}$, there exists a sequence $\eta = \{b_n: n \in \mathbb{N}\}$ in $X$ converging to $x$ and intersecting infinitely many members of $\xi$.

(6) $X$ is called an \textit{$\alpha_4$-space}, if for every point $x \in X$ and each sheaf $\{S_n: n \in \mathbb{N}\}$ with the vertex $x$, there exists a sequence converging to $x$ which meets infinitely many sequences $S_n$.

(7) $X$ is called an \textit{$\alpha_7$-space}, if for every point $x \in X$ and each sheaf $\{S_n: n \in \mathbb{N}\}$ with the vertex $x$, there exists a sequence converging to some point $y \in X$ which meets infinitely many sequences $S_n$.

(8) $X$ is said to have \textit{countable tightness} if for each $A \subseteq X$ and each $x \in \overline{A}$, there exists a countable subset $C \subset A$ such that $x \in \overline{C}$.
\end{definition}

\begin{definition}\cite{Arha}
Let $\zeta$ be a family of non-empty subsets of a space $X$.
\begin{enumerate}
    \item $\zeta$ is called a \textit{prefilter} on $X$ if whenever $P_1$ and $P_2$ are in $\zeta$, there exists $P \in \zeta$ such that $P \subset P_1 \cap P_2$.
    \item A prefilter $\zeta$ on $X$ is said to \textit{converge to a point} $x \in X$ if every open neighborhood of $x$ contains an element of $\zeta$.
    \item A prefilter $\zeta$ on $X$ is said to \textit{accumulate to a point} $x \in X$ if $x$ belongs to the closure of each element of $\zeta$.
    \item Two prefilters $\zeta$ and $\eta$ on $X$ are said to be \textit{synchronous} if, for any $P \in \zeta$ and $Q \in \eta$, $P \cap Q \neq \emptyset$.
    \item $X$ is called a \textit{bisequential space} if, for every prefilter $\zeta$ on $X$ accumulating to a point $x \in X$, there exists a countable prefilter $\eta$ on $X$ converging to the same point $x$ such that $\zeta$ and $\eta$ are synchronous.
    \item $X$ is called a \textit{biradial space} if, for every prefilter $\zeta$ on $X$ converging to a point $x \in X$, there exists a chain $\eta$ of subsets of $X$ converging to $x$ and synchronous with $\zeta$.
\end{enumerate}
\end{definition}

\begin{definition}
Let $\mathcal{P}$ be a family of subsets of a space $X$ with $x \in \bigcap \mathcal{P}$.

(1) The family $\mathcal{P}$ is called a \textit{network at $x$} if for each neighborhood $U$ of $x$ there exists $P \in \mathcal{P}$ such that $P \subset U$.

(2) The family $\mathcal{P}$ is called a \textit{cs-network at $x$} if for any sequence $L$ converging to $x$ and a neighborhood $U$ of $x$, there exists $P \in \mathcal{P}$ such that $L$ is eventually in $P$ and $P \subset U$.

(3) The family $\mathcal{P}$ is called an \textit{sn-network at $x$} if $\mathcal{P}$ is a network at $x$ and each element of $\mathcal{P}$ is a sequential neighborhood of $x$.

(4) The space $X$ is called \textit{csf-countable} (resp., \textit{snf-countable}) if for each $x \in X$, there is a countable cs-network (resp., sn-network) at $x$.\footnote{The terms "csf" and "snf" stand for "countable sequential fan" and "sequential neighborhood fan," respectively.}
\end{definition}

\begin{definition}\cite{Arha3}
Let $\mathcal{U}$ be a prefilter on a space $X$ and $x \in X$.
\begin{enumerate}
    \item $\mathcal{U}$ is called a \textit{nest} if $\mathcal{U}$ consists of open subsets of $X$ and has the following property: For any $U, V \in \mathcal{U}$, either $U \subset V$ or $V \subset U$.
    \item $X$ is \textit{$\pi$-nested at $x$} if there exists a nest in $X$ converging to $x$.
    \item $X$ is \textit{nested at $x$} if there is a nest in $X$ which forms a local base at $x$.
    \item $X$ is \textit{nested} if it is nested at each of its points.
\end{enumerate}
\end{definition}

\begin{definition} Let $X$ be a space and $x \in X$.

(1) The point $x$ is said to have an \textit{$\omega^\omega$-neighborhood base} or a \textit{local $\mathfrak{G}$} if there is a neighborhood base at $x$ of the form $\{U_\alpha(x): \alpha \in \mathbb{N}^\mathbb{N}\}$ such that $U_\beta(x) \subset U_\alpha(x)$ for all elements $\alpha \leq \beta$ in $\mathbb{N}^\mathbb{N}$, where $\mathbb{N}^\mathbb{N}$ consisting of all functions from $\mathbb{N}$ to $\mathbb{N}$ is endowed with the natural partial order, i.e., $f \leq g$ if and only if $f(n) \leq g(n)$ for all $n \in \mathbb{N}$.

(2) The space $X$ is said to have an \textit{$\omega^\omega$-base} if each $x \in X$, $x$ has an $\omega^\omega$-neighborhood base.
\end{definition}
Let us recall definitions of the \textit{sequential fan $S_\omega$} and \textit{Arens' space $S_2$}.

\begin{definition} (1) $S_\omega$ is the quotient space obtained from the disjoint sum of a countable family of convergent sequences via identifying limit points of all these sequences.

(2) $S_2 = (\mathbb{N} \times \mathbb{N}) \cup \{\infty\}$ is the space with each point of $\mathbb{N} \times \mathbb{N}$ isolated. The set $\{n\} \cup \{(m,n): m \geq k\}$ is a $k$-th basic neighborhood of $n \in \mathbb{N}$, and a set $U$ is a neighborhood of $\infty$ if and only if $\infty \in U$ and $U$ is an open neighborhood of all but finitely many $n \in \mathbb{N}$.
\end{definition}

\section{The biradial quotient spaces of strongly topological gyrogroups}
Drawing inspiration from the work of X. Ling and B. Zhao \cite{Ling}, who characterized biradial coset spaces, this section aims to explore and establish several properties pertaining to biradial quotient spaces of strongly topological gyrogroups.

\begin{lemma}\cite[Theorem 2]{Arha3}\label{lem3.3}
Every biradial space is $\pi$-nested at every point.
\end{lemma}

\begin{lemma}\label{lem3.2}
Let $G$ be a strongly topological gyrogroup and $H$ a closed $L$-subgyrogroup of $G$. Then for each open neighborhood $V$ at the identity there exists an open neighborhood $W$ at the identity such that $\ominus (W\oplus H)\subseteq H\oplus V$.
\end{lemma}
\begin{proof}
Since $G$ is a strongly topological gyrogroup, we can take a gyr-invariant symmetric open neighborhood $W$ at the identity such that $W\subseteq V$. For each $a\in W\oplus H$, there exist $w\in W$ and $h\in H$ such that $a=w\oplus h$.
Then
$$\ominus a=\ominus (w\oplus h)=\text{gyr}[w,h](\ominus h\ominus w)\in \text{gyr}[w,h](H\oplus W)=H\oplus W\subseteq H\oplus V.$$
\end{proof}

\begin{proposition}\label{prop3.3}
Let $H$ be a closed strong subgyrogroup of a strongly topological gyrogroup $G$.
If $H$ is neutral,
and \( G/H \) is \(\pi\)-nested at some point \( x \in G/H \), then \( G/H \) is nested.
\end{proposition}

\begin{proof}
Let $\pi: G \rightarrow G/H$ denote the natural quotient mapping and $0^* =\pi(0)$, where $0$ is the identity in $G$.
Given that $G/H$ is homogeneous, we can assume without loss of generality that $x = 0^*$. Consider a nest $\mathcal{A}$ that converges to $0^*$.
Define $\phi = \{\pi(\ominus T_A \oplus T_A) : A \in \mathcal{A}\}$, where for each $A \in \mathcal{A}$, $T_A = \pi^{-1}(A)$.
It is clear that $\phi$ forms a nest. We assert that $\phi$ serves as a local base at $0^*$ in $G/H$.
To substantiate this claim, let $O$ be any neighborhood of $0^*$. Then there is a neighborhood $U$ of $0$ such that $\pi(U)\subseteq O$. Since $H$ is neutral and $G$ is a strongly topological gyrogroup, there is a gyr-invariant neighborhood $V$ of $0$ such that $V\oplus V\subseteq U$ and $H\oplus(V\oplus V)\subseteq U\oplus H$. Since $H$ is a strong subgyrogroup, by Lemma \ref{lem3.2} for $V$ there is a neighborhood $W$ of $0$ such that $\ominus (W\oplus H)\subseteq H\oplus V$.
Since $\mathcal{A}$ converges to $0^*$, there exists $A \in \mathcal{A}$ such that $A \subseteq \pi(W)$. We claim that $\pi(\ominus T_A\oplus T_A )\subseteq O$, which implies that $\phi$ serves as a local base at $0^*$ in $G/H$. In fact,
\begin{align*}
&0 \in \ominus T_A\oplus T_A
\\&\subseteq \ominus(W\oplus H)\oplus(W\oplus H)
\\&\subseteq (H\oplus V)\oplus(W\oplus H)
\\&\subseteq (H\oplus V)\oplus(V\oplus H)
\\&=(H\oplus (V\oplus V))\oplus H  \quad\text{by Proposition \ref{prop2.10s}~}
\\&\subseteq(U\oplus H)\oplus H
\\&=U\oplus H.
\end{align*}
This implies $0^* \in \pi(\ominus T_A \oplus T_A) \subseteq \pi(U) \subset O$. Given the homogeneity of $G/H$, it follows that $G/H$ is nested.
\end{proof}

It is evident that every nested space is biradial. Utilizing Lemma \ref{lem3.3} and Proposition \ref{prop3.3} we derive the subsequent result.

\begin{theorem}\label{the3.7}
Let $H$ be a closed strong subgyrogroup of a strongly topological gyrogroup $G$.
If $H$ is neutral,
then \( G/H \) is biradial if and only if \( G/H \) is nested.
\end{theorem}

The symbol \( \chi(X) \) denotes the character of a space \( X \). It is defined as follows:
\[
\chi(X) = \sup\{\chi(x, X) : x \in X\},
\]
where
\[
\chi(x, X) = \min\{|\mathcal{B}_x| : \mathcal{B}_x \text{ is an open neighborhood base at } x \text{ of } X\} + \omega.
\]

\begin{lemma}\cite[Lemma10]{Arha3}\label{lem10}
Let \( X \) be a nested space and \( \chi(x, X) = \tau \) for every \( x \in X \). Then for any family \( \gamma \) of open subsets in \( X \) such that \( |\gamma| < \tau \), the set \(\bigcap\gamma \) is open.
\end{lemma}

A space \( X \) is called a \( P_\tau \)-space if for every family \( \gamma \) of open sets in \( X \) such that \( |\gamma| < \tau \), the set \( \bigcap \gamma \) is open. Considering that the quotient space is homogeneous, we obtain the following result from Theorem \ref{the3.7} and Lemma \ref{lem10}.

\begin{corollary}\label{corl3.9}
Let $H$ be a closed neutral strong subgyrogroup of a strongly topological gyrogroup $G$.
If \( G/H \) is biradial, then \( G/H \) is a \( P_\tau \)-space, where \( \tau = \chi(G/H) \).
\end{corollary}

 Theorem 1.1 and Corollary \ref{corl3.9} lead to the following conclusion:

\begin{corollary}\label{cor3.13}
Let $H$ be a closed neutral strong subgyrogroup of a strongly topological gyrogroup $G$.
If \( G/H \) is biradial, then either all \( G_\delta \)-subsets are open in \( G/H \), or \( G/H \) is metrizable.
\end{corollary}

\begin{corollary}
Let \( H \) be a closed neutral strong subgyrogroup of a strongly topological gyrogroup \( G \).
 Then \( G/H \) is metrizable if and only if \( G/H \) is biradial and all \( G_\delta \)-subsets are closed in \( G/H \).
\end{corollary}

It was demonstrated by Arhangel'ski\v{\i} \cite[Corollary 14]{Arha3} that a topological group \( G \) is metrizable if and only if \( G \) is bisequential. In the following, we generalize this result to quotient spaces of strongly topological gyrogroups by examining neutral strong subgyrogroups.

\begin{theorem}
Let \( H \) be a closed neutral strong subgyrogroup of a strongly topological gyrogroup \( G \).
Then \( G/H \) is metrizable if and only if \( G/H \) is bisequential.
\end{theorem}

\begin{proof}
The necessity is clear. We only need to prove the sufficiency. We can assume that \( G/H \) is non-discrete. Since \( G/H \) is sequential, it contains a non-trivial convergent sequence. Consequently, \( G/H \) must have a \( G_\delta \)-subset that is not open. By Corollary \ref{cor3.13}, \( G/H \) is metrizable.
\end{proof}

\begin{lemma}\label{lem3.8}
Let \( H \) be a closed strong subgyrogroup of a strongly topological gyrogroup \( G \).
If $H$ is neutral and
the nets $\{p(a_\delta) : \delta \in \Delta\}$ and $\{p(b_\delta): \delta \in \Delta\}$ of $G/H$ satisfy $\{p(a_\delta) : \delta \in \Delta\}$ and $\{p(b_\delta) : \delta \in \Delta\}$ both converging to $p(0)$, where $p : G \to G/H$ is the natural quotient mapping and $0$ is the identity in $G$,
then $\{p(\ominus a_\delta) : \delta \in \Delta\}$ and $\{p(a_\delta\oplus b_\delta) : \delta \in \Delta\}$ both converge to $p(0)$.
\end{lemma}

\begin{proof}
Let $O$ be an open neighborhood of $p(0)$. Then there exists an open neighborhood $U$ of $0$ such that $p(U\oplus U) \subseteq O$. Since $H$ is neutral, we can get an open neighborhood $V$ of $0$ such that $H\oplus V \subseteq U\oplus H$

First, we show that $\{p(\ominus a_\delta) : \delta \in \Delta\}$ converges to $p(0)$.

Since $G$ is a strongly topological gyrogroup and $H$ is a strong subgyrogroup, from Lemma \ref{lem3.2} it follows that there is an open neighborhood $W$ of $0$ such that $\ominus(W \oplus H) \subseteq H \oplus V$. Since \(p(a_\delta) \to p(0)\), there exists \(\beta \in \Delta\) such that \(p(a_\delta) \in p(W)\) whenever \(\delta \geq \beta\). Thus, \(\ominus a_\delta \in \ominus(W \oplus H) \subseteq H \oplus V \subseteq U \oplus H\) whenever \(\delta \geq \beta\).
Therefore, \(p(\ominus a_\delta) \in p(U) \subseteq O\) whenever \(\delta \geq \beta\). This implies that \(p(\ominus a_\delta) \to p(0)\).

Secondly, we show that $\{p(a_\delta\oplus b_\delta) : \delta \in \Delta\}$ converges to $p(0)$.

Since $G$ is a strongly topological gyrogroup, we can take a gyr-invariant open neighborhood $W$ of $0$ such that \(W \subseteq U \cap V\). Given that \(p(a_\delta) \to p(0)\) and \(p(b_\delta) \to p(0)\), there exists \(\beta \in \Delta\) such that \(p(a_\delta) \in p(W)\) and \(p(b_\delta) \in p(W)\)  whenever \(\delta \geq \beta\).

Note that $H$ and $W$ are gyr-invariant subsets. Then
\begin{align*}
a_\delta \oplus b_\delta  &\in (W\oplus H)\oplus(W\oplus H) \\
        &=(W\oplus(H\oplus W))\oplus H   \quad\text{by Proposition \ref{prop2.10s}~}\\
        &\subseteq (U\oplus(U\oplus H))\oplus H \quad \text{~by~}H\oplus W\subseteq H\oplus V\subseteq U\oplus H\\
        &=((U\oplus U)\oplus H)\oplus H. \quad  \quad\text{by Proposition \ref{prop2.10s}~}\\
        &=((U\oplus U)\oplus H.
\end{align*}
It follows that $p(a_\delta\oplus b_\delta) \in p(U\oplus U) \subseteq O$ whenever $\delta \geq \beta$. It implies that $p(a_\delta\oplus b_\delta) \to p(0)$.
\end{proof}

A space \( X \) is of \emph{countable pseudocharacter} if for each \( x \in X \), the set \( \{x\} \) is a \( G_\delta \)-subset of \( X \).

\begin{lemma}\cite[Proposition 5]{Arha3}\label{cor3.15}
If a space \( X \) of countable pseudocharacter is nested, then it is first-countable.
\end{lemma}

Applying Lemma \ref{cor3.15} and Theorems \ref{the3.7} and 1.1, we get the following result.

\begin{theorem}
Let \( H \) be a closed neutral strongly subgyrogroup of a strongly topological gyrogroup \( G \).
Then \( G/H \) is metrizable if and only if \( G/H \) is a biradial space with countable pseudocharacter.
\end{theorem}

\section{Quotient spaces of strongly topological gyrogroups with certain point-countable covers}
Bisequential spaces form a significant subclass within the broader category of strongly Fr\'{e}chet-Urysohn spaces. Nogura, Shakhmatov, and Tanaka \cite{Grue1} explored topological groups characterized by point-countable covers composed of bisequential spaces. In this section, we delve into the study of quotient spaces of strongly topological gyrogroups endowed with specific point-countable covers, focusing on the role of neutral subgyrogroups.

\begin{definition}
Let \(\mathcal{P}\) be a cover of a space \(X\).
\begin{enumerate}
\item The family \(\mathcal{P}\) is \textit{point-countable} if every point of \(X\) belongs to at most countably many elements of \(\mathcal{P}\).
\item The family \(\mathcal{P}\) is \textit{point-finite} if every point of \(X\) belongs to at most finitely many elements of \(\mathcal{P}\).
\end{enumerate}

A space \(X\) is said to possess the \textit{weak topology} relative to a cover \(\mathcal{P}\) if a subset \(F \subseteq X\) is closed in \(X\) precisely when its intersection \(F \cap P\) with each \(P \in \mathcal{P}\) is closed in \(P\). Notably, in this definition, the term "closed" can be equivalently substituted with "open." In line with \cite{Grue1}, we will employ the more concise phrase "\(X\) is determined by \(\mathcal{P}\)" to denote the condition that "\(X\) has the weak topology with respect to \(\mathcal{P}\)."

\end{definition}

\begin{lemma}\label{lem4.2}\cite[Lemma 2.5]{No}
Let \(X\) be a space determined by a point-countable cover \(\mathcal{P}\). Suppose that finite unions of elements of \(\mathcal{P}\) are either Fr\'{e}chet-Urysohn or sequential spaces in which each point is a \(G_\delta\)-set. If \(X\) contains neither a closed copy of \(S_\omega\) nor a closed copy of \(S_2\), then each point of \(X\) has an open neighborhood covered by some finite subfamily of \(\mathcal{P}\).
\end{lemma}

\begin{lemma}\label{lem4.3}\cite[Lemma 2.8]{No}
Let \(O\) be a nonempty open subset of a space \(X\), and let \(\mathcal{E}\) be a finite family of subsets of \(X\) such that \(O \subseteq \bigcup \mathcal{E}\). Then there exist \(E \in \mathcal{E}\) and a nonempty open subset \(B\) of \(X\) with \(B \subseteq \overline{B \cap E}\).
\end{lemma}

\begin{lemma}\cite[Fact 2.4 and Lemma 2.7]{No}\label{lem4.4}
If a space \(X\) is determined by a cover (resp., a point-finite cover) consisting of sequential spaces (resp., \(\alpha_4\)-spaces), then \(X\) is a sequential space (resp., an \(\alpha_4\)-space).
\end{lemma}

The subsequent proposition serves as the principal technical outcome of this section.

\begin{proposition}\label{pro4.5}
Let \( H \) be a closed neutral strongly subgyrogroup of a strongly topological gyrogroup \( G \).
Suppose that $G/H$ is determined
by a point-countable cover \(\mathcal{P}\) consisting of bisequential spaces.
Suppose that $G/H$ contains neither a closed copy of $S_{\omega}$ nor a closed copy of $S_2$.
Then $G/H$ is metrizable.
\end{proposition}

\begin{proof}
Let \( p : G \to G/H \) denote the natural quotient map, and let \( 0^* = p(0) \). We may assume that \( G/H \) is non-discrete.
Suppose \( \mathcal{P} \) is a closed point-countable cover that determines the topology of \( G/H \). By leveraging the closedness of all \( P \in \mathcal{P} \), it is straightforward to verify that \( \bigcup \mathcal{P}' \) is Fr\'{e}chet-Urysohn for every finite subfamily \( \mathcal{P}' \subseteq \mathcal{P} \).
So \( \mathcal{P} \) and \( G/H \) satisfy the conditions of Lemma \ref{lem4.2}. Consequently, there exist a finite subfamily \( \mathcal{E}\subset \mathcal{P} \) and an open subset \( O \) such that \(0^* \in O \subset \bigcup \mathcal{E}\).
By Lemma \ref{lem4.3}, there exist \( E \in \mathcal{E}\) and a non-empty open subset \( B \subset G/H \) satisfying \( B \subset \overline{B \cap E }\). Since \( B \) is non-empty, we can choose \( x \in B \cap E \). As an open subset of the bisequential space \( E \), the space \( B \cap E \) is bisequential in the induced topology. Hence, by Lemma \ref{lem3.3}, \( B \cap E \) is \(\pi\)-nested at \( x \).
According to \cite[Lemma 20]{Arha3}, \( \overline{B \cap E} \) is also \(\pi\)-nested at \( x \). Since \( x \in B \subset \overline{B \cap E} \), it follows that \( B \) is \(\pi\)-nested at \( x \). Finally, because \( B \) is open in \( G/H \), we conclude that \( G/H \) is \(\pi\)-nested at \( x \). By Proposition \ref{prop3.3}, \( G/H \) is nested.

By Lemma \ref{lem4.4}, \( G/H \) is sequential, and thus \( G/H \) contains a non-trivial convergent sequence. Consequently, \( G/H \) must have a \( G_{\delta} \)-subset that is not open. Since \( G/H \) is nested, it is first-countable by Lemma \ref{lem10}. It then follows from Theorem \ref{the1.1s} that \( G/H \) is metrizable.
\end{proof}

\begin{definition}\cite{Boo}
A mapping $f$ of a space $X$ onto a space $Y$ is called a  \textit{sequentially quotient mapping} if
whenever $\{y_n\}_{n \in \mathbb{N}}$ is a sequence converging to a point $y \in Y$ there are a convergent sequence $\{x_i\}_{i \in \mathbb{N}}$ in $X$
and a subsequence $\{y_{n_i}\}_{i \in \mathbb{N}}$ of $\{y_n\}_{n \in \mathbb{N}}$ with each $x_i \in f^{-1}(y_{n_i})$.
\end{definition}

\begin{proposition}\label{pro4.7}
Let $H$ be a closed neutral strong subgyrogroup of a strongly topological gyrogroup $(G, \tau, \oplus)$.
If the natural quotient mapping
$p$ of $G$ onto $G/H$ is a sequential quotient mapping, then the following conditions are equivalent.
\begin{enumerate}
    \item $G/H$ contains a closed copy of $S_2$.
    \item $G/H$ contains a closed copy of $S_\omega$.
\end{enumerate}
\end{proposition}

\begin{proof}
Let \( 0^* = p(0) \), where $0$ is the identity in $G$.
For each \( x \in G/H \), choose \( g_x \in G \) such that \( x = p(g_x) \).

(1)$\Rightarrow$(2). Since \( G \) and \( G/H \) are homogeneous and \( p \) is a sequential quotient map, we can assume without loss of generality that
\( A = \{0^*\} \cup \{p(a_n) : n \in \mathbb{N}\} \cup \{p(a_n(m)) : m, n \in \mathbb{N}\} \)
is a closed copy of \( S_2 \), where \( a_n \to 0\) as \( n \to \infty \) and \( a_n(m) \to a_n \) as \( m \to \infty \).
For each \( m, n \in \mathbb{N} \), define \( y_n(m) = p(a_n(m) \ominus a_n) \). Then, for each \( n \in \mathbb{N} \), we have \( y_n(m) \to 0^* \) as \( m \to \infty \). For every \( n \in \mathbb{N} \), let \( A_n = \{y_n(m) : m \in \mathbb{N}\} \).

\textbf{Claim 1.}
For each $m \in \mathbb{N}$, the set $F = \{n : A_m \cap A_n \text{ is infinite}\}$ is finite.

Otherwise, there exists \( m \in \mathbb{N} \) such that the set \( F = \{n : A_m \cap A_n \text{ is infinite}\} \) is infinite. Choose distinct elements \( p(a_{n_i}(m_i) \ominus a_{n_i}) \in A_m \cap A_{n_i} \) for \( n_i \in F \) with \( n_i < n_{i+1} \). Since \( p(a_{n_i}(m_i) \ominus a_{n_i}) \in A_m \) for each \( i \in \mathbb{N} \), it follows that \( p(a_{n_i}(m_i) \ominus a_{n_i}) \to 0^* \) as \( i \to \infty \).
Given that \( a_n \to 0 \) and \( a_{n_i}(m_i) \ominus a_{n_i} \to 0 \) as \( i \to \infty \), we have
\( p(a_{n_i} \oplus ((\ominus a_{n_i} \oplus (a_{n_i}(m_i) \ominus a_{n_i})) \oplus a_{n_i})) \to 0^* \)
as \( i \to \infty \).
By Proposition \ref{the3.22s}, we can get
$(a_{n_i}(m_i) \ominus a_{n_i}) \boxplus a_{n_i}=a_{n_i} \oplus ((\ominus a_{n_i} \oplus (a_{n_i}(m_i) \ominus a_{n_i})) \oplus a_{n_i})$.
So
\( p(a_{n_i} \oplus ((\ominus a_{n_i} \oplus (a_{n_i}(m_i) \ominus a_{n_i})) \oplus a_{n_i})) = p((a_{n_i}(m_i) \ominus a_{n_i}) \boxplus a_{n_i}) = p(a_{n_i}(m_i)), \)
it follows that \( p(a_{n_i}(m_i)) \to 0^* \) as \( i \to \infty \), which is a contradiction.

By Claim 1, without loss of generality, we can assume that \( A_i \cap A_j = \emptyset \) for distinct \( i, j \in \mathbb{N} \). Let
\( B = \{0^*\} \cup \{y_n(m) : n, m \in \mathbb{N}\}. \)

\textbf{Claim 2.}
The subspace $B$ of $G/H$ is a closed copy of $S_\omega$.

First, the set \( B \) is closed in \( G/H \). If not, there exists a point \( x \in (G/H) \setminus B \) such that \( x \in \overline{B} \).
Since \( A \) is closed, there exists a neighborhood $V$ of $0$ such that \( p(g_x \oplus V) \) intersects \(\{p(a_n(m)) : m, n \in \mathbb{N}\}\) for at most one \( n \). Since $G$ is a strongly topological gyrogroup, there is a gyr-invariant neighborhood $U$ of $0$ such that \( U \oplus U \subseteq V \).
Then, \( p(g_x \oplus U) \) contains an infinite subset \(\{y_{n_i}(m_i) : i \in \mathbb{N}\}\) of \( B \).
Since \( H \) is a neutral strong subgyrogroup of \( G \), there exists a gyr-invariant neighborhood $W$ of $0$
such that \( H \oplus W \subseteq U \oplus H \). Given that \( a_n \to 0 \) as \( n \to \infty \), we can assume that \(\{a_{n_i} : i \in \mathbb{N}\} \subseteq W \). Direct computation shows that
\[
\begin{split}
\{a_{n_i}(m_i) : i \in \mathbb{N}\} & =\{(a_{n_i}(m_i)\ominus a_{n_i})\boxplus a_{n_i}: i \in \mathbb{N}\} \\
&\subseteq ((g_x\oplus U)\oplus H)\boxplus W \\
& = \{ ((g_x \oplus u_1) \oplus h)\oplus \text{gyr}[(g_x \oplus u_1)\oplus h, \ominus u_2](u_2): u_1\in U, u_2 \in W, h \in H \} \\
& \subseteq ((g_x\oplus U)\oplus H)\oplus W \\
&= (g_x\oplus U)\oplus (H\oplus W)\quad\quad\quad\quad\quad \text{by Proposition \ref{prop2.10s} ~}\\
&= g_x\oplus (U\oplus(H\oplus W))\quad\quad\quad\quad\quad \text{by Proposition \ref{prop2.10s} ~}\\
&\subseteq g_x\oplus (U\oplus (U\oplus H))\quad\quad\quad\quad\quad  \text{~by~}H\oplus W\subseteq U\oplus H\\
&= g_x\oplus ((U \oplus U)\oplus H)\\
&\subseteq g_x\oplus (V\oplus H)\\
&= (g_x\oplus V)\oplus H. \quad \quad \quad\quad \quad \quad  (*)
\end{split}
\]
It follows that $\{p(a_{n_i}(m_i)) : i \in \mathbb{N}\} \subseteq p(g_x\oplus V)$, which is a contradiction.

If \( f: \mathbb{N} \to \mathbb{N} \), then the set \( C = \bigcup \{y_{n}(m) : m \leq f(n), n \in \mathbb{N}\} \) does not have any cluster point. Suppose, to the contrary, that there exists \( x \in \overline{C \setminus \{x\}} \). Since \( A \) is closed, we can choose \( V_1 \in \mathcal{U}\) such that \( \left| p(g_x \oplus V_1) \cap \{a_n(m) : m \leq f(n), n \in \mathbb{N}\} \right| \leq 1 \). Select \( U_1 \in \mathcal{U} \) such that \( U_1 \oplus U_1 \subseteq V_1 \).
Then, \( p(g_x \oplus U_1) \) contains an infinite subset \( \{y_{k_i}(l_i) : i \in \mathbb{N}\} \) of \( C \). Since \( H \) is a neutral strong subgyrogroup of \( G \), there exists \( W_1 \in \mathcal{U} \) such that \( H \oplus W_1 \subseteq U_1 \oplus H \).
Given that \( a_n \to 0 \) as \( n \to \infty \), we can assume that \( \{a_{k_i} : i \in \mathbb{N}\} \subseteq W_1 \). Following a similar calculation as before (*), we can deduce that \( \{a_{k_i}(l_i) : i \in \mathbb{N}\} \subseteq (g_x \oplus V_1) \oplus H \). Consequently, \( \{p(a_{k_i}(l_i)) : i \in \mathbb{N}\} \subseteq p(g_x \oplus V_1) \), which leads to a contradiction.

(2)$\Rightarrow$(1)
Let \( A = \{0^\ast\} \cup \{p(b_n(m)) : m, n \in \mathbb{N}\} \) be a closed copy of \( S_\omega \), where for each \( n \in \mathbb{N} \), \( b_n(m) \to 0 \) as \( m \to \infty \). Clearly, there exists a non-trivial sequence \( \{p(c_n)\}_{n \in \mathbb{N}} \) in \( G/H \) such that \( c_n \to 0 \) as \( n \to \infty \).
By Corollary \ref{cor2.17}, \( G/H \) is completely regular. Therefore, for each \( n \in \mathbb{N} \), let \( U_n \) be an open neighborhood of \( c_n \) such that \( \overline{p(U_i)} \cap \overline{p(U_j)} = \emptyset \) for distinct \( i, j \in \mathbb{N} \).
Define \( x_n = p(c_n) \) and \( x_n(m) = p(c_n \oplus b_n(m)) \) for each \( n, m \in \mathbb{N} \). For each \( n \in \mathbb{N} \), it follows that \( x_n(m) \to x_n \) as \( m \to \infty \). Without loss of generality, we can assume that \( \{x_n(m) : m \in \mathbb{N}\} \subset p(U_n) \).
Set \( B = \{0^\ast\} \cup \{x_n : n \in \mathbb{N}\} \cup \{x_n(m) : n, m \in \mathbb{N}\} \).

\textbf{Claim 3.} The subspace \(B\) of \(G/H\) is a closed copy of \(S_2\).

First, the set \( B \) is closed in \( G/H \). Indeed, if \( B \) were not closed, there would exist a point \( x \in (G/H) \setminus B \) such that \( x \in \overline{B} \).
Since \( A \) is closed, there exists a neighborhood $V$ of $0$ such that \( p(g_x \oplus V) \cap (A \setminus \{x\})= \emptyset \).
Since $G$ is a strongly topological gyrogroup, there is a gyr-invariant neighborhood $U$ of $0$ such that
\( U \oplus U \subseteq V \) and \( \overline{p(g_x \oplus U)} \cap \{x_n : n \in \mathbb{N}\} = \emptyset \). Clearly, for each \( n \in \mathbb{N} \), the set \( p(g_x \oplus U) \cap \{x_n(m) : m \in \mathbb{N}\} \) is finite.
Furthermore, \( p(g_x \oplus U) \) contains infinitely many elements of \( \{x_n(m) : n, m \in \mathbb{N}\} \), which we denote by \( \{x_{n_i}(m_i) : i \in \mathbb{N}\} \).
Since \( H \) is a neutral strong subgyrogroup of \( G \), there exists a gyr-invariant neighborhood $W$ of $0$
such that \( H \oplus W \subseteq U \oplus H \). Given that \( \ominus c_n \to 0\) as \( n \to \infty \), without loss of generality, we can assume that \( \{\ominus c_{n_i} : i \in \mathbb{N}\} \subseteq W \).
By a similar calculation as before, we have \( \{b_{n_i}(m_i) : i \in \mathbb{N}\} \subseteq (g_x \oplus V) \oplus H \). Hence, \( \{p(b_{n_i}(m_i)) : i \in \mathbb{N}\} \subseteq p(g_x \oplus V) \), which implies that \( p(g_x \oplus V) \) contains infinitely many elements of \( A \). This is a contradiction.

If \( f: \mathbb{N} \to \mathbb{N} \), using a similar argument as in Claim 2, we can show that \( \bigcup \{x_n(m) : m \leq f(n), n \geq k \text{ for some } k \in \mathbb{N}\} \) is closed in \( G/H \). Therefore, \( B \) is a closed copy of \( S_2 \).
\end{proof}

In conjunction with Propositions \ref{pro4.5} and \ref{pro4.7}, we obtain the following result.

\begin{corollary}
Let $H$ be a closed neutral strong subgyrogroup of a strongly topological gyrogroup $G$
such that the natural quotient mapping of \( G \) onto \( G/H \) is sequential quotient. Suppose that \( G/H \) is determined by a point-countable cover \( \mathcal{P} \) consisting of bisequential spaces. Assume also that:
\begin{enumerate}
    \item For each finite subfamily \( \mathcal{P}' \subseteq \mathcal{P} \), \( \bigcup\mathcal{P}' \) is Fr\'{e}chet-Urysohn.
    \item \( G/H \) does not contain a closed copy of \( S_\omega \).
\end{enumerate}
Then \( G/H \) is metrizable.
\end{corollary}

As is well known, every strongly Fr\'{e}chet-Urysohn space does not contain a closed copy of \( S_\omega \) or a closed copy of \( S_2 \).
By Proposition \ref{pro4.5},  we arrive at the following result.

\begin{proposition}
Let $H$ be a closed neutral strong subgyrogroup of a strongly topological gyrogroup $G$ such that \( G/H \) is determined by a point-countable cover \( \mathcal{P} \) consisting of closed bisequential spaces. Then the following conditions are
equivalent:
\begin{enumerate}
    \item \( G/H \) contains neither a closed copy of \( S_\omega \) nor a closed copy of \( S_2 \).
    \item \( G/H \) is metrizable.
\end{enumerate}
\end{proposition}

\begin{proof}
We only need to establish (1) \(\Rightarrow\) (2). Suppose \( \mathcal{P} \) is a closed point-countable cover that determines the topology of \( G/H \). By leveraging the closedness of all \( P \in \mathcal{P} \), it is straightforward to verify that \( \bigcup \mathcal{P}' \) is Fr\'{e}chet-Urysohn for every finite subfamily \( \mathcal{P}' \subseteq \mathcal{P} \). Consequently, Proposition \ref{pro4.5} ensures the metrizability of \( G/H \).
\end{proof}

\begin{theorem}\cite[Corollary 3.18]{Ba2}\label{the4.13}
Suppose that \( G \) is a strongly topological gyrogroup, \( H \) is a closed strong subgyrogroup of \( G \), and \( H \) is inner neutral, then the following are equivalent:
\begin{enumerate}
    \item \( G/H \) is first-countable;
    \item \( G/H \) is a bisequential space;
    \item \( G/H \) is a weakly first-countable space;
    \item \( G/H \) is a csf-countable and sequential \(\alpha_7\)-space.
\end{enumerate}
\end{theorem}

Combining with \cite[Lemmas 3.1 and 3.2]{LL} and Theorem \ref{the4.13}, we obtain the following result.

\begin{corollary}
Let $H$ be a closed neutral strong subgyrogroup of a strongly topological gyrogroup $G$.
Then \( G/H \) is metrizable if and only if \( G/H \) has no closed copy of \( S_\omega \) and \( G/H \) is determined by a point-countable cover consisting of first-countable spaces.
\end{corollary}

\section{Quotient spaces of strongly topological gyrogroups with Certain Local Countable Networks}

Recently, Ling et al. \cite{Ling} explored weak countability axioms in the context of coset spaces. In this section, we further investigate related results, particularly proving that if \( H \) is a closed neutral strong subgyrogroup of a strongly topological gyrogroup $G$, then

\begin{enumerate}
    \item \( G/H \) is metrizable \(\Leftrightarrow\) \( G/H \) is an \(snf\)-countable \( k \)-space with countable pseudocharacter;
    \item Suppose \( G/H \) has Baire property, if \( G/H \) has countable \(cn\)-character, then \( G/H \) is metrizable.
\end{enumerate}

\begin{lemma}\label{lem5.3}
Let $H$ be a closed neutral strong subgyrogroup of a strongly paratopological gyrogroup $G$. Suppose that
the sequences \(\{a_n\}_{n \in \mathbb{N}}\) and \(\{b_n\}_{n \in \mathbb{N}}\) satisfy \(p(a_n) \rightarrow p(0)\) and \(p(b_n) \rightarrow p(0)\),
where \(p : G \to G/H\) is the natural quotient mapping, then \(p(a_n \oplus b_n) \rightarrow p(0)\).
\end{lemma}

\begin{proof}
Let $O$ be an open neighborhood of $p(0)$, where 0
is the identity in $G$.
Since $G$ is a strongly topological gyrogroup, there is a gyr-invariant
neighborhood $U$ of 0 such that $p(U\oplus U) \subseteq O$.
Since $H$ is neutral, we can choose a gyr-invariant
neighborhood $V$ of 0such that $H\oplus V \subseteq U\oplus H$.
Take a gyr-invariant
neighborhood $W$ of 0 such that \(W \subseteq U \cap V\).
Given that \(p(a_n) \to p(0)\) and \(p(b_n) \to p(0)\), there exists
$m \in \mathbb{N}$ such that $p(a_n) \in p(W)$ and \(b_n \in p(W)\) for each $n \geq m$, i.e., $a_n \in W\oplus H$ and $b_n \in W\oplus H$.

Then
\begin{align*}
a_n \oplus b_n  &\in (W\oplus H)\oplus(W\oplus H) \\
        &=(W\oplus(H\oplus W))\oplus H  \\
        &\subseteq (U\oplus(U\oplus H))\oplus H \\
        &=((U\oplus U)\oplus H)\oplus H.
\end{align*}
It follows that $p(a_n\oplus b_n) \in p(((U\oplus U)\oplus H)\oplus H)=p(U\oplus U) \subseteq O$ for each $n \geq m$. So $p(a_n\oplus b_n) \to p(0)$.
\end{proof}

\begin{lemma}\label{lem5.4}
Let $H$ be a closed neutral strong subgyrogroup of a strongly paratopological gyrogroup $G$.
 Let $\{A_n : n \in \omega\}$ be a decreasing $sn$-network at $p(0)$ in $G/H$. Then for every $n \in \omega$, there is $m \in \omega$ such that
 $p(T_{A_m}\oplus T_{A_m}) \subseteq A_n$, where \(p : G \to G/H\) is the natural quotient mapping, $T_{A_m} = p^{-1}(A_m)$.
\end{lemma}

\begin{proof} Let $T_{A_n} = p^{-1}(A_n)$ for each $n \in \omega$. Suppose the contrary that there is an $n \in \omega$ such that $p(T_{A_m}\oplus T_{A_m}) \notin A_n$ for each $m \in \omega$. Then we can pick $a_m, b_m \in T_{A_m}$ with $p(a_m\oplus b_m) \notin A_n$ for each $m \in \omega$.
It follows from $\lim_{m \to \infty} p(a_m) = \lim_{m \to \infty} p(b_m) = p(0)$ and Lemma \ref{lem5.3} that $\lim_{m \to \infty} p(a_m\oplus b_m) = p(0)$.
This is a contradiction with the fact $A_n$ is a sequential neighborhood at $p(0)$ in $G/H$, which completes the proof.
\end{proof}

\begin{proposition}\label{pro5.5}\cite[Proposition 2.32]{JX3}
Let $(G, \tau,\oplus)$ be a paratopological gyrogroup and let $H$ be a $L$-subgyrogroup of $G$.
Then the canonical quotient mapping
$\pi$ from $G$ to its quotient topology on $G/H$ is an open and continuous mapping.
Moreover, if the subgyrogroup $H$ is a closed strong subgyrogroup,
 $G/H$ is a homogeneous $T_1$-space.
\end{proposition}

\begin{lemma}\label{lem5.4s}
Let \( H \) be a closed neutral strongly subgyrogroup of a strongly paratopological gyrogroup \( G \) such that the quotient space \( G/H \) is Hausdorff.
Suppose \(\{O_n(x) : n \in \omega, x \in G/H\}\) is a weak base on \( G/H \). Then, for every \( n \in \omega \),
put \(W_n(x) = p((g_x \oplus T_{O_n(0^*)}) \oplus T_{O_n(0^*)})\) where \( T_{O_n(x)} = p^{-1}(O_n(x)) \), \( p(0) = 0^* \), and \( p(g_x) = x \).
Then \(\{W_n(x) : n \in \omega, x \in G/H\}\) is a weak base on \(G/H\).
\end{lemma}

\begin{proof}
By Proposition \ref{pro5.5}, $G/H$ is a homogeneous $T_1$-space.
For each $x \in G/H$, fix $g_x\in G$ such that \( p(g_x) = x \).
Suppose further that, for each
$a\in G/H$, $h_{a}$ is a left translation of $G/H$ onto $G/H$ such that $h_{a}(b)=p({g_a}\oplus(g_b\oplus H))=p(({g_a}\oplus g_b)\oplus H)$.
Obviously, $h_{a}$ is a homeomorphism.

Put $U_n(y) = h_x(O_n(h_x^{-1}(y))\oplus O_n(h_x^{-1}(y)))=p((g_x \oplus T_{O_n(h_x^{-1}(y)}) \oplus T_{O_n(h_x^{-1}(y)})$, for each $y \in G/H$.
We can get $U_n(x) = W_n(x)$, for each $n \in \omega$.
By Lemma \ref{lem5.4} , the sequences $\{W_n(0) : n \in \omega\}$ and $\{O_n(0) : n \in \omega\}$ are cofinal, for each $n \in \omega$.
It follows that $\{W_n(0) : n \in \omega, \; x \in X\}$ is a weak base on $G/H$.
Since $h_x$ is a homeomorphism of $G/H$ onto $G/H$, it is clear that $\{W_n(x) : n \in \omega, \; x\in G/H\}$ is a weak base on $G/H$.
\end{proof}

\begin{lemma}\label{lem5.6}
Let \( H \) be a closed neutral strongly subgyrogroup of a strongly paratopological gyrogroup \( G \) such that the quotient space \( G/H \) is Hausdorff. Suppose \(\{O_n(x) : n \in \omega, x \in G/H\}\) is a weak base on \( G/H \). Then, for every \( n \in \omega \), there exists \( m \in \omega \) such that
\( p((g_x \oplus T_{O_m(0^*)}) \oplus T_{O_m(0^*)}) \subseteq O_n(x), \)
where \( T_{O_m(x)} = p^{-1}(O_m(x)) \), \( p(0) = 0^* \), and \( p(g_x) = x \).
\end{lemma}

\begin{proof}
By Proposition \ref{pro5.5}, $G/H$ is a homogeneous $T_1$-space.
Let \( T_{O_n(x)} = p^{-1}(O_n(x)) \) for each \( n \in \omega \) and each \( x \in G/H \).
Assume the contrary. Then, for some \( x \in G/H \) and some \( k \in \omega \), we have \( p((g_x\oplus T_{O_m(0^*)})\oplus T_{O_m(0^*)} \setminus O_k(x) \neq \emptyset \) for each \( m \in \omega \), and we fix \( x_m \in p((g_x\oplus T_{O_m(0^*)})\oplus T_{O_m(0^*)}) \setminus O_k(x) \).
Clearly, by Lemma \ref{lem5.4s} the sequence \( \eta = \{x_m : m \in \omega\} \) converges to \( x \).
Since \( G/H \) is Hausdorff, the set \( P = \{x_m : m \in \omega\} \cup \{x\} \) is closed in \( G/H \). Let \( P_0 = P \setminus \{x\} = \{x_m : m \in \omega\} \). Since \( \{O_n(x) : n \in \omega, \: x \in G/H\} \) is a weak base on \( G/H \), for each \( y \in G/H \setminus P \) there exists \( n(y) \in \omega \) such that \( O_{n(y)}(y) \cap P = O_{n(y)}(y) \cap P_0 =\emptyset\). We also have \( P_0 \cap O_k(x) = \emptyset \).
Since \( \{O_n(x) : n \in \omega, \: x \in G/H\} \) is a weak base on \(G/H \), it follows that the set \( P_0 \) is closed in \( G/H\), a contradiction with \( x\in \overline{P_0} \setminus P_0 \).
\end{proof}

\begin{theorem}\label{the5.3}
Let $H$ be a closed strong subgyrogroup of a paratopological gyrogroup $G$
such that \( G/H \) is Hausdorff. Then \( G/H \) is first-countable if and only if \( G/H \) is weakly first-countable.
\end{theorem}

\begin{proof}
The necessity is clear. We only need to prove the sufficiency.

Let \( (G, \tau, \oplus) \) be a paratopological gyrogroup, and let \( p : G \to G/H \) denote the natural quotient map, where \( 0^* = p(0) \). For each \( x \in G/H \), fix \( g_x \in G \) such that \( p(g_x) = x \), with \( g_{0^*} = 0 \). For a subset \( O \subseteq G/H \), define \( T_O = p^{-1}(O) \). Let \(\{O_n(x) : n \in \omega, x \in G/H\}\) be a weak base on \( G/H \). For each \( x \in G/H \) and \( n \in \omega \), let \( Q_n(x) = p((g_x \oplus T_{O_n(0^*)}) \oplus T_{O_n(0^*)}) \). By Lemma \ref{lem5.4s}, we may assume that \( O_n(x) = p(g_x \oplus T_{O_n(0^*)}) \) for each \( x \in G/H \) and each \( n \in \omega \).

We now show that, for each \( n \in \omega \), \( O_n(0^*) \) contains an open neighborhood of \( 0^* \) in \( G/H \). For each \( n \in \omega \), define
\begin{align*}
B_n = \{x \in O_n(0^*) : \text{ there exists } k \in \omega \text{ such that } p(g_x \oplus T_{O_k(0^*)}) \subseteq O_n(0^*)\}.
\end{align*}
Clearly, for each \( n \in \omega \), we have \( 0^* \in B_n \subseteq O_n(0^*) \). We claim that \( B_n \) is open in \( G/H \).

To prove this, take any \( y \in B_n \). By the definition of \( B_n \), there exists \( k \in \omega \) such that \( p(g_y \oplus T_{O_k(0^*)}) \subseteq O_n(0^*) \). By Lemma \ref{lem5.6}, there exists \( m \in \omega \) such that
$p((g_y \oplus T_{O_m(0^*)}) \oplus T_{O_m(0^*)}) \subseteq p(g_y \oplus T_{O_k(0^*)})$,
from which it follows that \( p((g_y \oplus T_{O_m(0^*)}) \oplus T_{O_m(0^*)}) \subseteq O_n(0^*) \). This implies that
$O_m(y) = p(g_y \oplus T_{O_m(0^*)}) \subseteq B_n.$
Since \( y \) was an arbitrary point of \( B_n \), it follows that \( B_n \) is open in \( G/H \).

Thus, \(\{B_n : n \in \omega\}\) forms a countable neighborhood base of \( G/H \) at \( 0^* \), which means that \( G/H \) is first-countable.
\end{proof}

\begin{fact}
\cite{Lin}\label{fact2.6}
Let $X$ be a space. Then
\begin{enumerate}
   \item $X$ is first-countable $\Leftrightarrow$ $X$ is Fr\'{e}chet-Urysohn and weakly first-countable.
    \item $X$ is weakly first-countable $\Leftrightarrow$ $X$ is sequential and $snf$-countable.
    \item $X$ is $snf$-countable $\Leftrightarrow$ $X$ is a $csf$-countable and $\alpha_4$-space.
    \item $X$ is strongly Fr¨¦chet-Urysohn $\Leftrightarrow$ $X$ is a Fr¨¦chet-Urysohn $\alpha_4$-space.
    \item $X$ is sequential $\Leftrightarrow$ $X$ is a $k$-space and every compact subset of $X$ is sequential
    \cite[Lemma 1.5]{Cha}.
\end{enumerate}
\end{fact}

Since a $k$-space with countable pseudocharacter is a sequential space, based on Fact \ref{fact2.6}(2) as well as Theorems \ref{the5.3} and \ref{the1.1s}, we can derive the following result.

\begin{corollary}
Let $H$ be a closed neutral strong subgyrogroup of a strongly topological gyrogroup $G$.
Then $G/H$ is metrizable if and only if $G/H$ is an snf-countable $k$-space with countable pseudocharacter.
\end{corollary}

A family $\mathcal{B}$ of open subsets of a space $X$ is called a {\it~$\pi$-base} of $X$ at a point $x \in X$ if all elements of $\mathcal{B}$ are nonempty and every open neighborhood of $x$ in $X$ contains an element of $\mathcal{B}$. As we know, every compact space $X$ with countable tightness has countable $\pi$-base at each point $x$ of $X$.

\begin{proposition}
Let $H$ be a closed neutral strong subgyrogroup of a strongly topological gyrogroup $G$ such that $G/H$ has a Hausdorff compactification $b(G/H)$ of countable tightness. Then $G/H$ is metrizable.
\end{proposition}

\begin{proof}
Consider the natural quotient mapping $p : G \to G/H$ and let $0^* = p(0)$, where 0
is the identity in $G$.
Given that $b(G/H)$ is compact and possesses countable tightness, it follows that $b(G/H)$ has a countable $\pi$-base at every point within $G/H$. Let $\mathcal{B}$ be a countable $\pi$-base at $0^*$ in $G/H$. For each $B \in \mathcal{B}$, define $T_B = p^{-1}(B)$.
Set $\gamma = \{p(\ominus T_B\oplus T_B) : B \in \mathcal{B}\}$. It can be shown that $\gamma$ forms an open neighborhood base of $0^*$ in $G/H$. This assertion holds because all elements of $\mathcal{B}$ are open sets in $G/H$ and include $0^*$. Suppose $O$ is an open neighborhood of $0^*$ in $G/H$.
Since $H$ is neutral and $G$ is a strongly topological gyrogroup, there exit gyr-invariant
neighborhoods $U, V, W$ of 0 such that $W \subseteq V$, $V\oplus V \subseteq U$, $H\oplus (V\oplus V)\subseteq U\oplus H$, and $p(U) \subseteq O$.
Since $\mathcal{B}$ serves as a $\pi$-base at $0^*$ in $G/H$, there exists some $B \in \mathcal{B}$ such that $B \subseteq p(W)$. Consequently, we have $T_B \subseteq W\oplus H$. Therefore,
\begin{align*}
&0 \in \ominus T_A\oplus T_A
\\&\subseteq \ominus(W\oplus H)\oplus(W\oplus H)
\\&\subseteq (H\oplus V)\oplus(W\oplus H)\quad\text{by Lemma \ref{lem3.2}~}
\\&\subseteq (H\oplus V)\oplus(V\oplus H)
\\&=(H\oplus (V\oplus V))\oplus H \quad\text{by Proposition \ref{prop2.10s}~}
\\&\subseteq(U\oplus H)\oplus H
\\&=U\oplus H.
\end{align*}
This implies that $0^* \in p(\ominus T_B\oplus T_B) \subseteq p(U) \subseteq O$. Consequently, $G/H$ is first-countable. Therefore, by Theorem \ref{the1.1s}, we can conclude that $G/H$ is metrizable.
\end{proof}

\begin{lemma}\cite[Lemma 1.5.22]{En89}\label{lem5.12}
    Let $f$ be an open continuous mapping from a space $X$ onto a space $Y$ and $B \subseteq Y$. Then $\overline{f^{-1}(B)} = f^{-1}(\overline{B})$.
\end{lemma}

Let $\Omega$ be a set and $I$ be a partially ordered set with an order $\leq$. We say that a family $\{A_i\}_{i \in I}$ of subsets of $\Omega$ is $I$-decreasing if $A_j \subseteq A_i$ for every $i \leq j$ in $I$. One of the most important example of partially ordered sets is the product $\mathbb{N}^{\mathbb{N}}$ endowed with the natural partial order, i.e., $\alpha \leq \beta$ if $\alpha_i \leq \beta_i$ for all $i \in \mathbb{N}$, where $\alpha = (\alpha_i)_{i \in \mathbb{N}}$ and $\beta = (\beta_i)_{i \in \mathbb{N}}$. For every $\alpha = (\alpha_i)_{i \in \mathbb{N}} \in \mathbb{N}^{\mathbb{N}}$ and each $k \in \mathbb{N}$, set $I_k(\alpha) = \{\beta \in \mathbb{N}^{\mathbb{N}} : \beta_i = \alpha_i \text{ for } i = 1,..., k\}$. In fact $I_k(\alpha)$ is completely defined by the finite subset $\{\alpha_1,...,\alpha_k\}$ of $\mathbb{N}$. So the family $\{I_k(\alpha) : k \in \mathbb{N}, \alpha \in \mathbb{N}^{\mathbb{N}}\}$ is countable.
\footnote{This stems from the fact that there exist only a finite number of ways to select a finite subset of $\mathbb{N}$ from $\mathbb{N}^{\mathbb{N}}$.} Let $\mathbf{M} \subseteq \mathbb{N}^{\mathbb{N}}$ and $\mathcal{U} = \{U_\alpha : \alpha \in \mathbf{M}\}$ be an $\mathbf{M}$-decreasing family of subsets of a set $\Omega$. Then we define the countable family $\mathcal{D}_\mathcal{U}$ of subsets of $\Omega$ by
\[
\mathcal{D}_\mathcal{U} = \{D_{k}(\alpha) : \alpha \in \mathbf{M}, k \in \mathbb{N}\}, \text{ where } D_{k}(\alpha) = \bigcap_{\beta \in I_k(\alpha) \cap \mathbf{M}} U_\beta,
\]
and say that $\mathcal{U}$ satisfies the \textit{condition} $\mathbf{D}$ if $U_\alpha = \bigcup_{k \in \mathbb{N}} D_k(\alpha)$ for every $\alpha \in \mathbf{M}$. Let $x$ be a point in a space $X$. We say that $X$ has a \textit{small base at $x$} if there exists an $\mathbf{M}_x$-decreasing base at $x$ for some $\mathbf{M}_x \subseteq \mathbb{N}^{\mathbb{N}}$.

A space $X$ has \textit{the strong Pytkeev property} \cite{Tsa} if for each $x \in X$, there exists a countable family $\mathcal{D}$ of subsets of $X$, such that for each neighborhood $U$ of $x$ and each $A \subseteq X$ with $x \in \overline{X \setminus A}$, there is $D \in \mathcal{D}$ such that $D \subseteq U$ and $D \cap A$ is infinite.
Gabriyelyan and Kakol demonstrated in \cite[Theorem 1.13]{Gab} that a Baire topological group is metrizable if and only if it possesses a countable $cn$-character. We aim to generalize this result to quotient spaces by focusing on neutral strong subgyrogroups.

\begin{theorem}
Let $H$ be a closed neutral strong subgyrogroup of a strongly topological gyrogroup $G$.
If $G/H$ has the Baire property, then the following are equivalent:
    \begin{enumerate}
        \item $G/H$ is metrizable.
        \item $G/H$ has the strong Pytkeev property.
        \item $G/H$ has countable $ck$-character.
        \item $G/H$ has countable $cn$-character.
        \item $G/H$ has an $\omega^\omega$-base satisfying the condition ($\mathbf{D}$).
    \end{enumerate}
\end{theorem}

\begin{proof}
It is evident that (1) $\Rightarrow$ (2) $\Rightarrow$ (3) $\Rightarrow$ (4) and (1) $\Rightarrow$ (5). By \cite[Theorem 1.12]{Gab}, we also have (5) $\Rightarrow$ (4). Therefore, it suffices to demonstrate that (4) $\Rightarrow$ (1).

Let $p: G \rightarrow G/H$ denote the natural quotient mapping, and $0^* = p(0)$, where 0
is the identity in $G$.
According to \cite[Theorem 1.12]{Gab}, there exists a small base $\{U_{\alpha} : \alpha \in \mathbf{M}\}$ at $0^*$ that satisfies the condition ($\mathbf{D}$). For each subset $K \subseteq G/H$, define $T_K = p^{-1}(K)$.
We aim to show that $\{p(\ominus \overline{T_{D_k(\alpha)}}\oplus \overline{T_{D_k(\alpha)}})^{\circ}: \alpha \in \mathbf{M}, k \in \mathbb{N}\}$ forms an open neighborhood base of $0^*$ in $G/H$. Let $O$ be an open neighborhood of $0^*$ in $G/H$.
Since $G$ is a strongly topological gyrogroup, there is a gyr-invariant
neighborhood $W$ of 0 such that $\overline{p(W)} \subseteq O$.
Since $H$ is neutral, there exists a gyr-invariant
neighborhood $V$ satisfying $H\oplus (V\oplus V) \subseteq W\oplus H$. Choose a gyr-invariant
neighborhood $V'$ such that $V' \subseteq W \cap V$. Consequently, there exists $\alpha \in \mathbf{M}$ for which $U_{\alpha} = \bigcup_{k \in \mathbb{N}} D_k(\alpha) \subseteq p(V'\oplus H)$.
Because $U_{\alpha}^{\circ}$ is open in $G/H$ and $G/H$ is a Baire space, there exists $k \in \mathbb{N}$ such that $(U_{\alpha}^{\circ} \cap \overline{D_k(\alpha)})^{\circ} \neq \emptyset$. This implies $\overline{D_k(\alpha)}^{\circ} \neq \emptyset$, and thus $p^{-1}(\overline{D_k(\alpha)}^{\circ}) \neq \emptyset$. By Lemma \ref{lem5.12}, we conclude that $\overline{T_{D_k(\alpha)}}^{\circ} \neq \emptyset$.
Therefore, $\ominus \overline{T_{D_k(\alpha)}}\oplus \overline{T_{D_k(\alpha)}}$ is a neighborhood of $0$ in $G$, and $p(\ominus \overline{T_{D_k(\alpha)}}\oplus \overline{T_{D_k(\alpha)}})$ is a neighborhood of $0^*$ in $G/H$.
Then
\begin{align*}
&\ominus T_{D_k(\alpha)}\oplus T_{D_k(\alpha)}
\\&\subseteq \ominus((V'\oplus H)\oplus H)\oplus((V'\oplus H)\oplus H)
\\&= \ominus(V'\oplus H)\oplus(V'\oplus H)
\\&\subseteq (H\oplus V)\oplus(V'\oplus H)\quad\text{by Lemma \ref{lem3.2}~}
\\&\subseteq (H\oplus V)\oplus(V\oplus H)
\\&=(H\oplus (V\oplus V))\oplus H  \quad\text{by Proposition \ref{prop2.10s}~}
\\&\subseteq(W\oplus H)\oplus H
\\&=W\oplus H.
\end{align*}
We can derive that $\overline{\ominus T_{D_k(\alpha)}}\oplus\overline{T_{D_k(\alpha)}} \subseteq \overline{W\oplus H}$. Consequently, we demonstrate that $0^* \in p(\overline{\ominus T_{D_k(\alpha)}}\oplus\overline{T_{D_k(\alpha)}})^{\circ} \subseteq \overline{p(W\oplus H)} \subseteq \overline{p(W)} \subseteq O$. Then $G/H$ is first-countable. By Theorem \ref{the1.1s}, it follows that $G/H$ is metrizable.
\end{proof}

\section{Acknowledgments}
We are immensely thankful to the reviewers for their extensive feedback and suggestions on our paper, and for their dedicated efforts to improve its overall quality.
\vskip0.9cm


\begin{thebibliography}{99}

\bibitem{Arha} A.V. Arhangel'ski\v{\i}, M. Tkachenko, Topological Groups and Related Structures, Atlantis Press and
World Sci., 2008.

\bibitem{Arha3} A.V. Arhangel'ski\v{\i}, On biradial topological spaces and groups, Topol. Appl. 36 (2) (1990) 173-180.

\bibitem{Arha2} A.V. Arhangel'ski\v{\i}, Classes of topological groups, Russ. Math. Surv. 36 (1981) 151-174.


\bibitem{Atip} W. Atiponrat, Topological gyrogroups: generalization of topological groups, Topol. Appl., 224 (2017), 73-82.
\bibitem{Atip1} W. Atiponrat, R. Maungchang, Complete regularity of paratopological gyrogroups, Topol. Appl., 270 (2020), 106951.



\bibitem{BL} M. Bao, F. Lin, Feathered gyrogroups and gyrogroups with countable pseudocharacter,
 Filomat, 33 (16) (2019), 5113-5124.


\bibitem{B1} M. Bao, X. Xu, A note on (strongly) topological gyrogroups, Topol. Appl., 307 (2022), 107950.

\bibitem{Ba1} M. Bao, X. Zhang, X. Xu, Topological gyrogroups with Fr\'{e}chet-Urysohn property and $\omega^{\omega}$-base, Bull. Iran. Math. Soc. 48
(2022) 1237-1248.

\bibitem{Ba2} M. Bao, X. Ling, X. Xu, Quotient spaces with strong subgyrogroups, arXiv:2204.02079v1.

\bibitem{Boo} J.R. Boone, F. Siwiec, Sequentially quotient mappings, Czechoslov. Math. J. 26 (1976) 174-182.

\bibitem{Bou} N. Bourbaki, \'{E}l\'{e}ments de Math\'{e}matique, Premiere Parite, Livre 3, 3-m ed., Actualiti\'{e}s Sci. et Ind., vol. 916, Hermann,
Paris, 1942.

\bibitem{Cai} Z. Cai, S. Lin, W. He, A note on Paratopological Loops, Bulletin of the Malaysian Math. Sci. Soc.,
42(5) (2019), 2535-2547.

\bibitem{Cha} M.J. Chasco, E. Mart\'{\i}n-Peinador, V. Tarieladze, A class of angelic sequential non-Fr\'{e}chet-Urysohn topological groups, Topol. Appl. 154 (2007) 741-748.


\bibitem{En89} R. Engelking, General Topology (revised and completed edition), Berlin: Heldermann Verlag,
1989.


\bibitem{Fer} M. Fern\'{a}ndez, I. S\'{a}nchez, M.G. Tkachenko, Coset spaces and cardinal invariants, Acta Math. Hung. 159 (2) (2019)
486-502.


\bibitem{Fo} T. Foguel, A.A. Ungar, Involutory decomposition of groups into twisted subgroups and subgroups, J. Group Theory 3
(2000) 27-46.


\bibitem{Gab} S.S. Gabriyelyan, J. Kakol, On topological spaces and topological groups with certain local countable networks, Topol.
Appl. 190 (2015) 59-73.


\bibitem{Grue1} G. Gruenhage, E. Michael, Y. Tanaka, Spaces determined by point-countable covers, Pac. J. Math. 113 (1984) 303-332.


\bibitem{Hea} R.W. Heath, Monotone normality in topological groups, Zb. Rad. Filozofskog Fak. Ni\v{s}u, Ser. Math. 3 (1989) 13-18.
\bibitem{Hea1} R.W. Heath, D.J. Lutzer, P. Zenor, Monotonically normal spaces, Trans. Am. Math. Soc. 178 (1973) 481-493.

\bibitem{Hr}M. Hru\u{s}\'{a}k, U.A. Ramos-Garc\'{\i}a, Malykhin's problem, Adv. Math. 262 (2014) 193-212.

\bibitem{JX2} Y. Jin, L. Xie, H. Yang. On the continuity of the inverse in (strongly) paratopological gyrogroups, Filomat, 38(15) (2024), 5449-5462.

\bibitem{JX3}Y. Jin, L. Xie, P. Yan. Three-space properties in paratopological gyrogroups, Filomat, 39(4) (2025), 1181-1196.

\bibitem{Lin}S. Lin, Point-Countable Covers and Sequence-Covering Mappings, second edition, Science Press, Beijing, 2015 (in Chinese).

\bibitem{Ling} X. Ling, B. Zhao. Weak countability axioms of coset spaces, Topology Appl., 2024, 341: 108751.

\bibitem{LL} C. Liu, S. Lin, Generalized metric spaces with algebraic structures, Topology and its Applications, 157 (2010) 1966-1974.

\bibitem{Liu1} L. Liu, J. Zhang, The first-countability in the quotient spaces of topological gyrogroups, Topology and its Applications, 2023, 329: 108473.

\bibitem{Liu} X. Liu, C. Liu, S. Lin, Strict Pytkeev networks with sensors and their applications in topological groups, Topol. Appl. 258
(2019) 58-78.


\bibitem{Mo} J.T. Moore, S. Todor\v{c}evi\'{c}, The metrization problem for Fr\'{e}chet groups, in: E. Pearl (Ed.), Open Problems in Topology
II, Elsevier, 2007, 201-206.

\bibitem{Pa} J. Park, S. Kim, Hilbert projective metric on a gyrogroup of qubit density matrices, Rep. Math. Phys. 76 (3) (2015)
389-400.


\bibitem{Suk} T. Suksumran, K. Wiboonton, Lagrange's theorem for gyrogroups and the Cauchy property, Quasigr. Relat. Syst. 22 (2)
(2014) 283-294.

\bibitem{No} T. Nogura, D.B. Shakhmatov, Y. Tanaka, Metrizability of topological groups having weak topologies with respect to good covers, Topol. Appl. 54 (1993) 203-212.


\bibitem{Ny} P.J. Nyikos, Metrizability and the Fr\'{e}chet-Urysohn property in topological groups, Proc. Am. Math. Soc. 83 (4) (1981)
793-801.




\bibitem{Sh}R.X. Shen, On generalized metrizable properties in quasitopological groups, Topol. Appl. 173 (2014) 219-226.


\bibitem{ST} T. Suksumran, The Algebra of Gyrogroups: Cayley¡¯s Theorem, Lagrange¡¯s Theorem, and Isomorphism Theorems, Springer
International Publishing, Cham, (2016), 369-437.

\bibitem{Suk3} T. Suksumran, K. Wiboonton, Isomorphism theorems for gyrogroups and $L$-subgyrogroups, J. Geom. Symmetry Phys., 37
(2015), 67-83.


\bibitem{Tsa} B. Tsaban, L. Zdomskyy, On the Pytkeev property in spaces of continuous functions (II), Houst. J. Math. 35 (2009)
563-571.


\bibitem{Ung} A.A. Ungar, Analytic Hyperbolic Geometry and Albert Einstein's Special Theory of Relativity, World Scientific, 2008.

\bibitem{Va} J.E. Vaughan, Small uncountable cardinals and topology, in: J. van Mill, G.M. Reed (Eds.), Open Problems in Topology,
North-Holland, Amsterdam, 1990, 195-218.


\bibitem{zh} J. Zhang J, K. Lin K, W. Xi, Weakly first-countability in strongly topological gyrogroups, Topology and its Applications, 2024, 350: 108920.
\end{thebibliography}
\end{document}